\documentclass[12pt]{amsart}
\usepackage{amssymb}
\usepackage{amsmath, amsthm}
\usepackage{amsthm,amssymb,amstext,amscd,amsfonts,amsbsy,amsxtra,latexsym,amsmath,mathtools}
\usepackage{multirow}
\usepackage[normalem]{ulem} 

\usepackage[margin=1in]{geometry}

\usepackage{hyperref}
\hypersetup{
  colorlinks   = true, 
  urlcolor     = blue, 
  linkcolor    = blue, 
  citecolor   = red 
}

\newtheorem{thm}{Theorem}[section]
\newtheorem{lemma}[thm]{Lemma}

\newtheorem{theorem}[thm]{Theorem}

\theoremstyle{remark}

\newtheorem*{remarks}{Remarks}
\newtheorem*{rmk*}{Remark}
\newtheorem*{OpenProblem}{Open Problem}

\newcommand{\Z}{\mathbb{Z}}

\newcommand{\HH}{\mathbb{H}}
\newcommand{\R}{\mathbb{R}}
\newcommand{\C}{\mathbb{C}}

\newcommand{\GL}{\operatorname{GL}}

\newcommand{\SL}{\operatorname{SL}}
\newcommand{\SO}{\operatorname{SO}}

\newcommand{\colvector}[2]{\left(\begin{smallmatrix} #1 \\ #2 \end{smallmatrix}\right)}
\newcommand{\ColVector}[2]{\left(\begin{array}{c} #1 \\ #2 \end{array}\right)}

\DeclareMathOperator{\sgn}{sgn}
\makeatletter
\let\@@pmod\pmod
\DeclareRobustCommand{\pmod}{\@ifstar\@pmods\@@pmod}
\def\@pmods#1{\mkern4mu({\operator@font mod}\mkern 6mu#1)}
\makeatother

\numberwithin{equation}{section}

\begin{document}

\title{On a relation between certain $q$-hypergeometric series and Maass waveforms}
\author{Matthew Krauel, Larry Rolen, and Michael Woodbury}
\address{Mathematical Institute\\University of
Cologne\\ Weyertal 86-90 \\ 50931 Cologne \\Germany}
\email{mkrauel@math.uni-koeln.de} 
\email{woodbury@math.uni-koeln.de}
\address{212 McAllister Building \\
The Pennsylvania State University\\
University Park, PA 16802}
\email{larryrolen@psu.edu}

\date{\today}
\thanks{
\indent The first author is supported by the European Research Council (ERC) Grant agreement n. 335220 - AQSER. The second author thanks the University of Cologne and the DFG for their generous support via the University of Cologne postdoc grant DFG Grant D-72133-G-403-151001011, funded under the Institutional Strategy of the University of Cologne within the German Excellence Initiative.
}

\maketitle

\begin{abstract}
In this paper, we answer a question of Li, Ngo, and Rhoades concerning a set of $q$-series related to the $q$-hypergeometric series $\sigma$ from Ramanajun's lost notebook. Our results parallel a theorem of Cohen which says that $\sigma$, along with its partner function $\sigma^*$, encode the coefficients of a Maass waveform of eigenvalue $1/4$. 
\end{abstract}

\section{Introduction}

The function 
\[
\sigma(q):=\sum_{n\geq0}\frac{q^{\frac{n(n+1)}2}}{(-q)_n}
,
\]
where $(a;q)_n=(a)_n:=\prod_{j=0}^{n-1}\left(1-aq^j\right)$ and $|q|<1$, was first considered in Ramanujan's 
``Lost'' notebook \cite{AndrewsLostNotebookV}. Andrews, Dyson, and Hickerson showed \cite{AndrewsDysonHickerson} that this series satisfies several striking and beautiful properties, and in particular that if 
\[
\sigma(q)=:\sum_{n\geq0}S(n)q^n
,
\]
then $\lim \sup |S(n)|=\infty$ but $S(n)=0$ for infinitely many $n$. Their proof was closely related to showing that $\sigma$ can be written as the indefinite theta series
\[
\sigma(q)=\sum_{\substack{n\geq0\\ |j|\leq n}}(-1)^{n+j}q^{\frac{n(3n+1)}2-j^2}\left(1-q^{2n+1}\right)
.
\]
Cohen \cite{Cohen} then introduced the complementary function 
\[
\sigma^*(q):=2\sum_{n\geq1}\frac{(-1)^nq^{n^2}}{(q;q^2)_n}
\]
and used $\sigma$ and $\sigma^*$ to nicely package the work of Andrews, Dyson, and Hickerson within a single modular object. 
Namely, he showed that if we define coefficients $\{T(n)\}_{n\in24\Z+1}$ by 
\[
q\sigma\left(q^{24}\right)
=:\sum_{n\geq0}T(n)q^n
,
\]
\[
q^{-1}\sigma^*\left(q^{24}\right)=:\sum_{n<0}T(n)q^{-n}
,
\]
then $T(n)$ are the Fourier coefficients of a Maass waveform. To describe this, set $\tau = {x}+i{y} \in\mathbb{H}$ and $e(w):=e^{2\pi i w}$ for a variable $w$.  Then Cohen proved that for $K_0$, a modified Bessel function of the second kind, the function 
\begin{equation}\label{eq:Cohensu}
u(\tau)
:=
y^{\frac{1}{2}}\sum_{n\in24\Z+1}T(n)K_0\left(\frac{2\pi |n|{y}}{24}\right)e\left(\frac {nx}{24}\right)
\end{equation}
is a Maass waveform on the congruence subgroup $\Gamma_0(2)$. That is, $u$ satisfies the transformations
\begin{align*}
u\left(-\frac1{2\tau}\right)
&=\overline{u(\tau)}
, \\
u(\tau+1)&=e\left(\frac1{24}\right)u(\tau)
,
\end{align*}
and is an eigenfunction of the hyperbolic Laplacian 
\[
\Delta:=-{y}^2\left(\frac{\partial^2}{\partial {x}^2}+\frac{\partial^2}{\partial {y}^2}\right)
\]
with eigenvalue $1/4$.

Subsequently, further examples of related series were considered by a number of authors.  For example, Corson, Favero, Liesinger, and Zubairy \cite{CFLZ-Characters} defined the pair of $q$-series
\begin{align*}
W_1(q):&=\sum_{n\geq0}\frac{(q)_n(-1)^nq^{\frac{n(n+1)}2}}{(-q)_n}
,
\\
W_2(q):&=\sum_{n\geq1}\frac{(-1;q^2)_n(-1)^nq^n}{(q;q^2)_n}
,
\end{align*}
and established many identities for them.  These $q$-series identities inspired Bringmann and Kane \cite{BK-Multiplicative} and Lovejoy \cite{Lovejoy} to study further examples of related $q$-series.

Meanwhile, Li, Ngo, and Rhoades \cite{LNR-Renorm2}, in the context of Maass waveforms and in relation to Zagier's theory of quantum modular forms \cite{Zagier}, studied the functions $W_1(q)$ and $W_2(q)$.  In particular, using the identities of \cite{CFLZ-Characters}, they showed that $W_1(q)$ and $W_2(q)$ encode the coefficients of a Maass waveform of eigenvalue $1/4$ in a similar manner to that of $\sigma$ and $\sigma^*$ as described above.

Noting similarities between the $q$-series studied in \cite{BK-Multiplicative} and \cite{Lovejoy} with $\sigma$, $\sigma^*$ and $W_1$, $W_2$ surrounding connections to the arithmetic of real quadratic fields, Li, Ngo, and Rhoades asked if the series
\begin{align*}
f_1(q)&:=\sum_{n\geq0}\frac{q^{\frac{n(n+1)}2}}{(-q)_n\left(1-q^{2n+1}\right)}, &f_2(q)&:=\sum_{n\geq1}\frac{q^{\frac{n(n+1)}2}}{(-q)_{n-1}\left(1-q^{2n-1}\right)},\\
f_3(q)&:=\sum_{n\geq0}\frac{(q)_{2n}q^n}{(-q)_{2n+1}}, 
&f_4(q)&:=\sum_{n\geq0}\frac{(q)_{2n+1}q^{n+1}}{(-q)_{2n+2}},\\
f_5(q)&:=\sum_{n\geq0}\frac{(-1)^n(q)_nq^{\frac{n(n+1)}2}}{(q;q^2)_{n+1}}, 
&f_6(q)&:=\sum_{n\geq1}\frac{(-1)^n\left(q^2;q^2\right)_{n-1}q^n}{(q^n;q)_n},\\
f_7(q)&:=\sum_{n\geq0}\frac{(-1)^n\left(q^2;q^2\right)_{n}q^{n^2+n}}{(-q)_{2n+1}}, 
&f_8(q)&:=\sum_{n\geq1}\frac{(q)_{n-1}q^n}{\left(-q^n\right)_n},\\
LL(q)&:=\sum_{n\geq1}\frac{(-1)^n(q)_{n-1}q^{\frac{n(n+1)}2}}{(-q)_n},\qquad \text{ and}
&L(q)&:=\sum_{n\geq1}\frac{\left(q^2;q^2\right)_{n-1}q^n}{\left(-q^2;q^2\right)_n}
\end{align*}
can also be described in terms of Maass wave forms.

\begin{OpenProblem}[Li, Ngo, Rhoades, \cite{LNR-Renorm2}]
Relate the series $f_1, f_2,\ldots, f_8, L, LL$ to indefinite theta functions. 
\end{OpenProblem}

The authors state that these examples ``do not seem to fit as nicely into the theory of Maass waveforms.'' They explain that the difficulty of showing their relation to Maass waveforms stems from not knowing whether they are associated to Hecke characters, information available in the cases of $\sigma, \sigma^*$ and $W_1,W_2$. As a result, it is unclear if Cohen's method of proof applies to these examples. In this paper, we resolve Li, Ngo, and Rhoades' open problem and circumvent this obstruction.

In order to state our result for any Maass waveform $f(\tau)$ of eigenvalue $1/4$ on a congruence subgroup, say with 
Fourier expansion
\[
f(\tau)
=:
y^{\frac{1}{2}}\sum_{n\neq0}A(n)K_{0}\left(2\pi |n|{y}\right)e(n {x})
,
\]
we define the $q$-series associated to its positive and negative coefficients by 
\[
f^{\pm}(\tau):=\sum_{\sgn(n)=\pm1}A(n)q^n
.
\]
Note that this formal $q$-series will, with appropriate growth conditions on $A(n)$, converge on the half-plane $\{\tau\in\C : \sgn(\operatorname{Im}(\tau))=\pm1\}$.
We remark in passing that such a map from Maass forms to $q$-series was studied extensively by 
Lewis and Zagier \cite{LewisZagier1,LewisZagier2}, and used by Zagier \cite{Zagier} to show that 
$\sigma,\sigma^*$ give rise to a quantum modular form. Our main result is the following theorem. 

\begin{theorem}\label{mainthm}
Let $g$ be any of the series $f_1,f_2,\ldots,f_8,L,LL$. Then there is a component $F_g$ of a vector-valued Maass waveform on $\operatorname{SL}_2(\Z)$ of eigenvalue $1/4$ such that $F_g^+ = C + q^\alpha g$ for constants $\alpha$ and $C$ given in the following Table \ref{tableThm}.
\begin{table}[!h]
\begin{center}
\addtolength{\tabcolsep}{1mm}
\renewcommand{\arraystretch}{1.2}
\begin{tabular}{||c||c|c|c|c|c|c|c|c|c|c|}
\hline
$g$ & $f_1$ & $f_2$ & $f_3$  & $f_4$ & $f_5$ & $f_6$ & $f_7$ & $f_8$ & $LL$ & $L$  \\
\hline
$\alpha$   & $\frac{1}{16} $ & $-\frac{7}{16}$  & $\frac{1}{2}$ & $0$ & $\frac{1}{4}$ & $-\frac{1}{4}$ & $\frac{1}{3}$ & $-\frac{1}{3}$ & $0$ & $0$ \\
\hline
$C$ & $0$ & $0$ & $0$  & $-\frac{1}{4}$ & $0$ & $0$ & $0$ & $0$ & $-\frac{1}{4}$ & $-\frac{1}{4}$  \\
\hline
\end{tabular}\\[1mm]
\caption{Constants relating $F_g^+$ with $g\in \{f_1 ,\dots ,f_8, LL, L\}$.}
\label{tableThm}
\end{center}
\end{table}
\end{theorem}

\pagebreak

\begin{remarks} 
We make the following comments.

\begin{enumerate}
\item We recall that Cohen showed that $\sigma^*$ corresponds to the negative coefficients of the Maass waveform attached to $\sigma$, and that $\sigma (q^{-1}) =-\sigma^* (q)$ for roots of unity $q$. Based on similar relations, (see \cite{LNR-Renorm2}) $f_3 (q^{-1}) =f_3(q)$, $f_4(q^{-1})=-f_4(q)$, $f_5(q^{-1}) =-f_6(q)$, $f_7(q^{-1}) =-f_8(q)$, and $L(q^{-1})=LL(q)$, Li, Ngo, and Rhoades suggest that it is natural to expect that the map $q\mapsto q^{-1}$ is intimately related with the negative coefficients in general. Indeed, after checking with a computer we suspect it can be shown that $F_{f_3}^- (q^{-1}) = -F_{f_3}^+ (q)$, $F_{f_4}^- (q^{-1}) = F_{f_4}^+ (q)$, $F_{f_5}^- (q^{-1}) = F_{f_6}^+ (q)$, $F_{f_6}^- (q^{-1}) = F_{f_5}^+ (q)$, and $F_{f_7}^- (q^{-1}) = e(-1/3)F_{f_8}^+ (q)$, while other identities presumably exist as well. The general phenomenon behind these relations remains a mysterious direction meriting further investigation.
 \item
 Using the methods discussed in \cite{LNR-Renorm2}, it is possible to show that the functions in Theorem \ref{mainthm} are quantum modular forms on the sets where the $q$-hypergeometric series terminate. It would be interesting to pursue this further, and to find the maximal possible ``quantum sets'' (which, as mentioned in \cite{LNR-Renorm2}, is related to the vanishing of the corresponding Maass waveform at cusps).
 \item
Another intriguing direction would be to identify the other components of the vector-valued Maass waveforms arising in Theorem \ref{mainthm} as $q$-hypergeometric series.
 \end{enumerate}
\end{remarks}

The paper is organized as follows. In Section \ref{PrelimSection}, we review the necessary preliminaries, and in particular the work of Zwegers \cite{Zwegers-Maass} which allows us to study the functions in Theorem \ref{mainthm}.   In Section \ref{ProofsSection}, we prove Theorem \ref{mainthm} by first rewriting these $q$-series in terms of indefinite theta functions and then applying Zwegers' theory discussed in Section \ref{PrelimSection}.

\section*{Acknowledgements}

The first and third authors would like to thank Kathrin Bringmann
for suggesting the work on this problem as part of an ERC project,
and the authors thank her also for many enlightening conversations. The authors also thank the referee for many helpful comments.

\section{Preliminaries}\label{PrelimSection}

In this section we summarize the relevant work of Zwegers \cite{Zwegers-Maass} and set notation. Suppose $A$ is a symmetric $2\times 2$ matrix with integral coefficients such that the quadratic form 
\begin{equation}\label{eq:Q}
 Q({\nu}) := \frac12 {\nu}^T A {\nu} 
\end{equation}
is indefinite, where ${\nu}^T$ denotes the transpose of ${\nu}$.  Let $B({\nu},\mu)$ be the associated bilinear form given by
\begin{equation}\label{eq:B}
 B({\nu},\mu) := {\nu}^T A \mu = Q({\nu}+\mu)-Q({\nu})-Q(\mu),
\end{equation}
and $c_1,c_2\in \R^2$ be vectors such that $Q(c_j)=-1$ for $j=1,2$ and $B(c_1,c_2)<0$.  In particular, this implies that $c_1$ and $c_2$ belong to the same component of the set $C_Q:=\{ c\in \R^2\mid Q(c)=-1\}$.  

Let us choose $c_0\in C_Q$ and set
 \[ C_Q^+ := \{ c\in C_Q\mid B(c,c_0)<0\} \qquad\mbox{and}\qquad
    C_Q^- := \{ c\in C_Q\mid B(c,c_0)>0\}. \]
Note that these are the two components of $C_Q$.  Moreover, there exists a unique $P\in \GL_2(\R)$ such that
 \[ A = P^T \left(\begin{array}{cc} 0 & 1 \\ 1 & 0 \end{array}\right) P \]
and $P^{-1}\colvector{1}{-1}= c_0$.  With this $P$ in place, for each $c\in C_Q^+$ there is a unique $t \in \mathbb{R}$ such that
 \begin{equation}\label{eq:c(t)}
 c = c(t):= P^{-1}\left( \begin{array}{c} e^t \\ -e^{-t} \end{array} \right). 
 \end{equation}
In other words, we have an explicit parametrization of $C_Q^+$.  Additionally, given $c\in C_Q^+$ we let $c^\perp =c^\perp(t):=P^{-1}\colvector{e^t}{e^{-t}}$.  Note that $B(c,c^\perp)=0$, and $Q(c^\perp)=1$.  It is easily seen that these two conditions determine $c^\perp$ up to sign.

Set 
\begin{align*}
\rho_B({r})=\rho_{B}^{c_1,c_2}({r}):=&\frac{1}{2} \left[1-\sgn \left(B\left({r},c_1\right)B\left({r},c_2\right)\right) \right] \hspace{5mm} \text{and} \\ 
\rho_B^\perp ({r})=\rho_{B}^{c_1^\perp,c_2^\perp} ({r}):=&\frac{1}{2}\left[1-\sgn\left(B\left({r},c_1^\perp\right)B\left({r},c_2^\perp\right)\right)\right].
\end{align*}
 Given $c_j=c(t_j)\in C_Q^+$ for $j=1,2$, Zwegers defined the function 
 \begin{equation}\label{ZwegersPhiDefn}
\begin{aligned}
 \Phi_{a,b}(\tau)=\Phi_{a,b}^{c_1,c_2}(\tau)  :&= \sgn(t_2-t_1) {y}^{\frac{1}{2}} \sum_{{r} \in a+\Z^2} \rho_B({r}) e( Q({r}){x}+ B({r},b))K_0(2\pi Q({r}){y}) \\
 & \quad + \sgn(t_2-t_1) {y}^{\frac{1}{2}} \sum_{{r} \in a+\Z^2} \rho_B^\perp ({r}) e( Q({r}){x}+B({r},b))K_0(-2\pi Q({r}){y})
\end{aligned}
\end{equation}
where again $\tau={x}+i{y}$ is in the upper half-plane $\HH$, $q=e(\tau)$, and $K_0$ is the Bessel function as in the introduction and which satisfies $(x\frac{\partial ^2}{\partial {x}^2}+\frac{\partial}{\partial {x}}-{x})K_0({x})=0$. 

As indicated in \cite{Zwegers-Maass}, it is not difficult to see that for particular choices of the parameters $Q$, $a$, $b$, $c_1$ and $c_2$, $\Phi_{a,b}(\tau)$ is Cohen's Maass waveform \eqref{eq:Cohensu}.  In general, assuming convergence, it is immediate from the differential equation satisfied by $K_0$ that for an arbitrary choice of parameters, $\Phi_{a,b}(\tau)$ is an eigenvector of the Laplace operator $\Delta$ with eigenvalue $1/4$.  Zwegers found a certain completion of $\Phi_{a,b}(\tau)$ which, as described below, transforms like a modular form.  Moreover, conditions are given under which it is shown to be true that $\Phi_{a,b}$ is equal to its completion.  To describe this, we first consider the series
 \[ 
 \varphi_{a,b}^c(\tau) := {y}^{\frac{1}{2}}\sum_{{r}\in a+\Z^2} \alpha_{t} \left({r} {y}^{\frac{1}{2}} \right) q^{Q({r})}e(B({r},b))
 , 
 \]
with 
 \[ 
 \alpha_{t}({r}):= 
\begin{cases} \displaystyle{\int}_{t}^\infty e^{-\pi B({r},c(u))^2}du & \mbox{ if }B\left({r},c\right)B\left({r},c^\perp\right)>0, \\[2ex]
 -\displaystyle{\int}_{-\infty}^{t} e^{-\pi B({r},c(u))^2}du & \mbox{ if }B\left({r},c\right)B\left({r},c^\perp \right)<0, \\
 0 & \mbox{ otherwise. }
 \end{cases}
%
 \]
where $t$ satisfies \eqref{eq:c(t)}. These functions satisfy the following transformation properties.
\begin{lemma}[Zwegers \cite{Zwegers-Maass}] \label{Zlem}
For $c\in C_Q^+$ and $a,b\in\R^2$, let $\varphi_{a,b}^c$ be defined as above.  Then
\[  \varphi_{a+\lambda,b+\mu}^c = e(B(a,\mu))\varphi_{a,b}^c\quad\mbox{for all $\lambda\in \Z^2$ and }\mu\in A^{-1}\Z^2, \]
\[  \varphi_{-a,-b}^c = \varphi_{a,b}^c, \]
and
\[  \varphi_{\gamma a,\gamma b}^{\gamma c} = \varphi_{a,b}^c \quad \mbox{for all }\gamma\in \SO^+(Q,\Z). \]
Here 
 \[ \SO^+(Q,\Z):=\left\{\gamma \in \SL_2(\Z)\mid Q(\gamma {r})=Q({r})\mbox{ for all }{r}\in \R^2, \gamma(C_Q^+)=C_Q^+\right\}. \]
\end{lemma}

Zwegers' main result (reformulated slightly for our purposes) is as follows.

\begin{theorem}[Zwegers \cite{Zwegers-Maass}] \label{Zthm}
The function $\Phi_{a,b}(\tau)$ is well defined (i.e.,~converges absolutely) for any choice of parameters $a,b$, and $Q$.  Moreover, the function
\begin{equation}\label{eq:Phihat}
 \widehat{\Phi}_{a,b} = \Phi_{a,b}+\varphi_{a,b}^{c_1}-\varphi_{a,b}^{c_2}
\end{equation}
also converges absolutely and satisfies 
 \[ \widehat{\Phi}_{a+\lambda,b+\mu} = e(B(a,\mu))\widehat{\Phi}_{a,b}\quad \mbox{for all $\lambda\in \Z^2$ and }\mu\in A^{-1}\Z^2,\]
 \[ \widehat{\Phi}_{-a,-b} = \widehat{\Phi}_{a,b},\]
and the modular relations
\begin{align*}
  \widehat{\Phi}_{a,b}(\tau+1) & = e\left(-Q(a)-\frac12 B\left(A^{-1}A^*,a\right)\right)\widehat{\Phi}_{a,a+b+\frac12 A^{-1}A^*}(\tau), \\
  \widehat{\Phi}_{a,b}\left(-\frac{1}{\tau}\right) & = \frac{e(B(a,b))}{\sqrt{-\det{A}}} \sum_{p\in A^{-1}\Z^2\pmod*{\Z^2}} \widehat{\Phi}_{-b+p,a}(\tau),
\end{align*}
where $A^*$ is the vector comprised of the diagonal entries of $A$.
\end{theorem}

\begin{rmk*}
Strictly speaking, Zwegers' function $\widehat{\Phi}_{a,b}$ (the aforementioned completion of $\Phi_{a,b}$) is not defined as in \eqref{eq:Phihat}.  Rather, $\widehat{\Phi}_{a,b}$ is defined in a more direct way, and then the content of Theorem~\ref{Zthm} (which is a combination of Theorems~2.4 and 2.6 in \cite{Zwegers-Maass}) is that his completion agrees with the right hand side of \eqref{eq:Phihat}.
\end{rmk*}

Roughly speaking, for each of the examples in Theorem~\ref{mainthm} we will find parameters $Q$, $a$, $b$, $c_1$ and $c_2$ that realize each of the functions as the ``positive part'' of $\Phi_{a,b}$ (corresponding to the first sum on the right hand side of \eqref{ZwegersPhiDefn}).  This is done in Section~\ref{sec:proofa}. In Section~\ref{sec:proofb}, we use Theorem~\ref{Zthm} and Lemma~\ref{Zlem} to prove that $\Phi_{a,b}$ is modular, hence a Maass form.


\section{Proofs}\label{ProofsSection}

We first show how each of the series in Theorem \ref{mainthm} may be written in terms of an indefinite theta series of the shape studied by Zwegers.

\subsection{Representations of indefinite theta series}\label{sec:proofa}

Note that in each of the statements and proofs of the results in this section a matrix $A$ is given which corresponds to a quadratic form $Q=Q_A$ and a bilinear form $B=B_A$ as in \eqref{eq:Q} and \eqref{eq:B}, respectively.  We trust that $Q$ and $B$ are clear from context without specifically referring back to the matrix $A$. We also note that although each case will be worked out in detail, at the end of this section the reader may find a convenient table summarizing all of the choices of parameters made in this section.
\begin{lemma}\label{lem:f1}
Setting $A= \left(\begin{smallmatrix} 8&0 \\0&-4 \end{smallmatrix} \right)$, 
  \begin{equation}
  \notag
  a_1 =  \begin{pmatrix} \frac{5}{8} \\ \frac{1}{2} \end{pmatrix}, \hspace{3mm} a_2 =  \begin{pmatrix} \frac{1}{8} \\ 0 \end{pmatrix}, \hspace{3mm} b=\begin{pmatrix} 0 \\ 0 \end{pmatrix} ,\hspace{3mm} c_1 = \begin{pmatrix} -\frac{1}{2} \\ 1 \end{pmatrix}, \hspace{3mm}  \text{and} \hspace{3mm} c_2 = \begin{pmatrix} \frac{1}{2} \\ 1 \end{pmatrix},
  \end{equation}
we have 
  \begin{equation}\label{eq:f1info2}
  q^{\frac{1}{16}} f_1 (q) =  \sum_{\ell =1}^2  \sum_{{r} \in a_\ell +\mathbb{Z}^2} \rho_B({r})q^{Q({r})} e(B({r} ,b)).
  \end{equation}
\end{lemma} 
 
\begin{proof}
Proposition $3.1$ of \cite{BK-Multiplicative} states
 \begin{equation}\label{eq:BKf1}
 f_1 (q) = \sum_{-n-1 \leq j \leq n} q^{4n^2 +5n +1 -2j^2 -2j}\left( 1+q^{6n+6}\right) 
   + \sum_{-n \leq j \leq n} q^{4n^2 +n -2j^2}\left( 1+q^{6n+3}\right).
 \end{equation}
Focusing first on the left sum we note that
\begin{equation}\label{eq:lemf1a}
 4n^2 +5n +1 -2j^2 -2j = 4\Big(n+\frac{5}{8}\Big)^2 -2\Big(j+\frac{1}{2}\Big)^2 - \frac{1}{16}.
\end{equation}
Making the change of variables $n \mapsto -n-2$, \eqref{eq:lemf1a} becomes $4n^2 +11n +7 -2j^2 -2j$. This same change of variables, moreover, transforms the set over which the summation is applied via
 \[  \left\{ \ColVector{n}{j} \in \Z^2 \left\vert \begin{array}{c} -1 \leq n+j, \\ 0\leq n-j\end{array}\right. \right\} \mapsto 
     \left\{ \ColVector{n}{j} \in \Z^2\left\lvert \begin{array}{c}n-j < 0, \\  n+j < -1 \end{array}\right. \right\} .
\]
Therefore, 
 \begin{align}\label{eq:lemf1d}
 \sum_{-n-1 \leq j \leq n} q^{4n^2 +5n +1 -2j^2 -2j}\left( 1+q^{6n+6}\right) 
 &= q^{-\frac{1}{16}} \left ( \sum_{\substack{-1\leq n+j \\ 0\leq n-j}} + \sum_{\substack{n-j \leq -1\\ n+j\leq 0}}\right ) q^{4 \left(n+ \frac{5}{8} \right)^2 -2\left(j+ \frac{1}{2} \right)^2}.
 \end{align}
To find a $\Phi_{a_1,b}$ which coincides with this expression, it is clear that we must choose $a_1$ as in the statement of the Lemma.  Then for $r_1 = \colvector{n}{j}+a_1$, we see that
\begin{align}\label{eq:lemf1c}
  \rho_B\left({r}_1 \right) &=\frac{1}{2}\left[1+\sgn \left(\left(4(n+j)+\frac{9}{2}\right)\left(4(n-j)+\frac{1}{2}\right) \right) \right] \\ \nonumber
 &= \begin{cases} 1 &\text{if } n+j\geq -1 \text{ and } n-j\geq 0 \text{ or } n+j< -1 \text{ and } n-j< 0, \\ 0 & \text{otherwise.} \end{cases}
\end{align}
In other words, $\rho_B(\colvector{n}{j}+a_1)$ equals, as a function of $n$ and $j$, the characteristic function of the exact set over which the summation on the right side of \eqref{eq:lemf1d} is applied.

We can treat the right sum of \eqref{eq:BKf1} similarly.  In this case,
\begin{equation}\label{eq:lemf1b}
 4n^2 +n -2j^2  = 4\Big(n+\frac{1}{8}\Big)^2 -2j^2 -\frac{1}{16}.
\end{equation}
The transformation $n \mapsto -n-1$ maps \eqref{eq:lemf1b} to $4n^2 +7n +3 -2j^2$ and the set
 \[ \left\{ \left.\ColVector{n}{j}\in \Z^2\right\vert \begin{array}{c} 0 \leq n+j, \\ 0 \leq n-j \end{array} \right\} \mapsto \left\{\ColVector{n}{j}\in \Z^2 \left\vert \begin{array}{c} n+j<0, \\ n-j<0 \end{array}\right. \right\}.\]
It follows that
\begin{align}\label{eq:lemf1e}
 \sum_{-n \leq j \leq n} q^{4n^2 +n -2j^2}\left( 1+q^{6n+3}\right) 
 &= q^{-\frac{1}{16}} \left( \sum_{\substack{ 0\leq n+j \\ 0\leq n-j}} + \sum_{\substack{  n-j <0 \\ n+j < 0}}\right) q^{4 \left(n+ \frac{1}{8} \right)^2 -2j^2}.
 \end{align}
As occurred in \eqref{eq:lemf1c}, for $r_2=\colvector{n}{j}+a_2$, $\rho_B(r_2)$ (as a function of $n$ and $j$) is the characteristic function of precisely the set over which the summation on the right hand side of \eqref{eq:lemf1e} is applied.

Therefore, in order to get \eqref{eq:f1info2}, we simply add \eqref{eq:lemf1d} and \eqref{eq:lemf1e}.  With the choice of $b=\colvector{0}{0}$, since this would imply that $B(r_\ell,b)=0$ for both $\ell=1$ and $\ell=2$, this gives the desired result.
\end{proof}


\begin{lemma}\label{lem:f2}
Setting $A= \left(\begin{smallmatrix} 8&0 \\0&-4 \end{smallmatrix} \right)$, 
  \begin{equation}
  \notag
  a_1 =  \begin{pmatrix} \frac{3}{8} \\ 0 \end{pmatrix}, \hspace{3mm} a_2 =  \begin{pmatrix} \frac{1}{8} \\ \frac{1}{2} \end{pmatrix}, \hspace{3mm} b=\begin{pmatrix} 0 \\ 0 \end{pmatrix} ,\hspace{3mm} c_1 =  \begin{pmatrix} -\frac{1}{2} \\ 1 \end{pmatrix}, \hspace{3mm}  \text{and} \hspace{3mm} c_2 =  \begin{pmatrix} \frac{1}{2} \\ 1 \end{pmatrix},
  \end{equation}
we have 
  \begin{equation}\label{eq:f2info2}
  q^{-\frac{7}{16}} f_2 (q) =  \sum_{\ell =1}^2 \sum_{{r} \in a_\ell +\mathbb{Z}^2} \rho_B({r})q^{Q({r})} e(B({r} ,b)).
  \end{equation}
 \end{lemma}

\begin{proof}
Proposition~3.2 of \cite{BK-Multiplicative} states
 \begin{equation} \label{eq:BKf2}
 f_2 (q) = \sum_{\substack{n\geq 1 \\ -n \leq j \leq n-1}} q^{4n^2 -n -2j^2 -2j}\left( 1+q^{2n}\right) 
   + \sum_{\substack{n\geq 0 \\ -n \leq j \leq n}} q^{4n^2 +3n +1 -2j^2}\left( 1+q^{2n+1}\right).
 \end{equation}
 In the left sum we find 
 \begin{equation} \label{eq:f2eq1}
 4n^2 -n -2j^2 -2j = 4\left(n-\frac{1}{8}\right)^2 -2\left(j+\frac{1}{2}\right)^2 +\frac{7}{16},
 \end{equation}
  while under the change of variables $n \mapsto -n$, \eqref{eq:f2eq1} becomes $4n^2 +n -2j^2 -2j = 4(n+1/8)^2 -2(j+1/2)^2 +7/16$. We pause to note that instead of choosing $a_2$ as we have, we could have taken it to contain the entry $-1/8$ (instead of $+1/8$) corresponding to \eqref{eq:f2eq1}, as well as other values. We have chosen our $a_2$ here, and the $a$ terms throughout this paper, to have entries of smallest nonnegative values. However, either value of $-1/8$ or $+1/8$ in $a_2$ leads to the $\ell=2$ term in \eqref{eq:f2info2}. 

Meanwhile, in the right sum of \eqref{eq:BKf2} we have 
  \[
  4n^2 +3n +1 -2j^2  = 4\left(n+\frac{3}{8}\right)^2 -2j^2 +\frac{7}{16},
  \]
   and taking $n \mapsto -n-1$ produces $4n^2 +5n +2 -2j^2$. A similar analysis as in the proof of Lemma \ref{lem:f1} gives the $\ell=1$ term in \eqref{eq:f2info2}. 
\end{proof}


 Unlike the previous two functions, $f_3$ is comprised of one sum which must first be doubled in order to be written in the desired form.

\begin{lemma}\label{lem:f3}
Setting $A= \left(\begin{smallmatrix} 4&0 \\0&-2 \end{smallmatrix} \right)$,
  \begin{equation}
  \notag
  a =  \begin{pmatrix} \frac{1}{2} \\ 0 \end{pmatrix}, \hspace{5mm} b=\begin{pmatrix} 0 \\ \frac{1}{4} \end{pmatrix} ,\hspace{5mm}  c_1 = \frac{1}{\sqrt{2}} \begin{pmatrix} -1 \\ 2 \end{pmatrix},\hspace{3mm}  \text{and}\hspace{3mm} c_2 = \frac{1}{\sqrt{2}} \begin{pmatrix} 1 \\ 2 \end{pmatrix},
  \end{equation}
  we have
  \begin{equation}\label{eq:f3info2}
   q^{\frac{1}{2}} f_3 (q) = \frac{1}{2} \sum_{{r} \in a +\mathbb{Z}^2} \rho_B({r})q^{Q({r})} e(B({r} ,b)).
  \end{equation} 
 \end{lemma}

\begin{proof}
 Consider the quadratic form 
 \[ 2n^2 +2n -j^2 = 2\left(n+\frac12\right)^2 - j^2 -\frac12.\]
Then
 \begin{equation}
 \notag
 f_3 (q) = \sum_{\substack{n\geq 0 \\ -n \leq j \leq n}} (-1)^j q^{2n^2 +2n - j^2} =  \frac{1}{2} \left(\sum_{\substack{0 \leq n+j \\ 0 \leq n-j}} + \sum_{\substack{ n+j < 0 \\ n-j < 0}}\right) (-1)^j q^{2 \left(n+ \frac{1}{2} \right)^2 - j^2},
 \end{equation}
 where the first equality is Proposition $3.4$ of \cite{BK-Multiplicative} and the second follows from the taking the change of variables $n \mapsto -n-1$ on one-half of the sum, which preserves $2n^2 +2n -j^2$.
 
 Suppose ${r} = \left(\begin{smallmatrix} n\\ j \end{smallmatrix} \right)+\colvector{\frac12}{0}$. Thus $e(B({r} ,b)) = (-1)^{j}$ for all $n,j \in \mathbb{Z}$, and we obtain $\rho_B({r})$ along with the remainder of the proof of \eqref{eq:f3info2} in a similar way as in the previous lemmas.
 \end{proof}

 The functions $f_4$ and $f_5$ (as well as $LL(q)$ and $L(q)$ below) emit theta series when placed in the context of Zwegers' work.

\begin{lemma}\label{lem:f4}
Setting $A= \left(\begin{smallmatrix} 4&0 \\0&-2 \end{smallmatrix} \right)$,
  \begin{equation}
  \notag
  a =  \begin{pmatrix} 0 \\ 0 \end{pmatrix}, \hspace{5mm} b=\begin{pmatrix} 0 \\ \frac{1}{4} \end{pmatrix} ,\hspace{5mm} c_1 = \frac{1}{\sqrt{2}} \begin{pmatrix} -1 \\ 2 \end{pmatrix},  \hspace{3mm}  \text{and} \hspace{3mm}c_2 = \frac{1}{\sqrt{2}} \begin{pmatrix} 1 \\ 2 \end{pmatrix},
  \end{equation}
  we have
  \begin{equation}\label{eq:f4info2}
   f_4 (q) -\frac{1}{4} = -\frac{1}{2} \sum_{{r} \in a +\mathbb{Z}^2} \rho_B({r})q^{Q({r})} e(B({r} ,b)).
  \end{equation}
 \end{lemma}

\begin{proof}
 Proposition $3.5$ of \cite{BK-Multiplicative} states
 \begin{equation}
 \notag
 f_4 (q) = -\sum_{\substack{n\geq 0 \\ -n-1 \leq j \leq n}} (-1)^j q^{2n^2 +4n +2 - j^2} .
 \end{equation}
 Note 
 \[ 2n^2 +4n +2 -j^2 = 2(n+1)^2 - j^2, \]
and the map $n \mapsto -n-2$ preserves $2n^2 +4n +2 -j^2$. Therefore, just as in the previous lemma we find
 \begin{align*}
 -2f_4 (q) &= \left(\sum_{\substack{n\geq 0 \\ -n-1 \leq j \leq n}} +\sum_{\substack{-n+1\geq 0 \\ n+2 \leq j \leq -n-2}}\right) (-1)^j q^{2(n+1)^2 - j^2} 
 &= \left(\sum_{\substack{0 \leq n+j \\ 1 \leq n-j}} + \sum_{\substack{ n+j < 0 \\ n-j < 1}}\right) (-1)^j q^{2n^2 - j^2},
 \end{align*}
 where we also applied $n\mapsto n-1$ to obtain the last equality.
 
  Suppose ${r} = \left(\begin{smallmatrix} n \\ j \end{smallmatrix} \right)$ for $n,j \in \mathbb{Z}$. Then $e(B({r} ,b)) = (-1)^{j}$ and
 \begin{align*}
 2\rho_B({r}) &=\left[1-\sgn\left(-8\left(n+j\right)\left(n-j\right) \right) \right] \\
 &= \begin{cases} 2 &\text{if } n+j\geq 1 \text{ and } n-j\geq 1 \text{ or } n+j< 0 \text{ and } n-j< 0, \\ 1& \text{if } n+j =0 \text{ or } n-j=0,\\ 0 & \text{otherwise.} \end{cases}
 \end{align*}
 Equation \eqref{eq:f4info2} now follows from
 \begin{equation}
 \notag
 \sum_{\substack{n+j =0 \\n-j \geq 1}} (-1)^j q^{2n^2 -j^2} = \sum_{n\geq 1} (-1)^n q^{n^2} =\sum_{\substack{n-j =0 \\n+j < 0}} (-1)^j q^{2n^2 -j^2}.
 \end{equation}
 \end{proof}
 

\begin{lemma}\label{lem:f5}
  Setting $A= \left(\begin{smallmatrix} 3&0 \\0&-1 \end{smallmatrix} \right)$,
  \begin{equation}
  \notag
  a =  \begin{pmatrix} \frac{1}{2} \\ \frac{1}{2} \end{pmatrix}, \hspace{5mm} b=\begin{pmatrix} 0 \\ 0 \end{pmatrix} , \hspace{5mm} c_1 = \frac{1}{\sqrt{3}} \begin{pmatrix} -1 \\ 3 \end{pmatrix},  \hspace{3mm}  \text{and} \hspace{3mm}c_2 = \frac{1}{\sqrt{3}} \begin{pmatrix} 1 \\ 3 \end{pmatrix},
  \end{equation}
  we have
  \begin{equation}\label{eq:f5info2}
  q^{\frac{1}{4}} f_5 (q) = \frac{1}{2} \sum_{{r} \in a +\mathbb{Z}^2} \rho_B({r})q^{Q({r})} e(B({r} ,b)).
  \end{equation}
 \end{lemma}

\begin{proof}
 Proposition $3.6$ of \cite{BK-Multiplicative} states
 \begin{equation}
 \notag
 f_5 (q) = \sum_{-n \leq j \leq n} q^{\frac{3}{2}\left(n^2 +n\right) - \frac{1}{2}\left(j^2 +j\right)} .
 \end{equation}
 Note 
 \[
 \frac{3}{2}\left(n^2 +n\right) - \frac{1}{2}\left(j^2 +j\right) = \frac{3}{2}\left(n+\frac{1}{2}\right)^2 -\frac{1}{2}\left( j+\frac{1}{2}\right)^2 -\frac{1}{4},
 \]
  while the change of variables $n \mapsto -n-1$ preserves $n^2 +n$.  Using that
 \begin{equation}\notag
 \sum_{\substack{n-j =0 \\n+j \geq 0}} q^{\frac{3}{2}\left(n+\frac{1}{2}\right)^2 -\frac{1}{2}\left( j+\frac{1}{2}\right)^2} = \sum_{n\geq 0} q^{\left(n+\frac{1}{2}\right)^2} \quad \text{and} \quad \sum_{\substack{n+j =-1 \\n-j < 0}} q^{\frac{3}{2}\left(n+\frac{1}{2}\right)^2 -\frac{1}{2}\left( j+\frac{1}{2}\right)^2} =\sum_{n\leq -1} q^{\left(n+\frac{1}{2}\right)^2} ,
 \end{equation}
 a similar analysis as in Lemma \ref{lem:f4} establishes \eqref{eq:f5info2}.
 \end{proof}


\begin{lemma}\label{lem:f6}
  Setting $A= \left(\begin{smallmatrix} 6&0 \\0&-2 \end{smallmatrix} \right)$,
  \begin{equation}
  \notag
  a_1 =  \begin{pmatrix} \frac{1}{2}  \\ 0 \end{pmatrix}, \hspace{3mm} a_2 = \begin{pmatrix} 0  \\ \frac{1}{2} \end{pmatrix}, \hspace{3mm} \hspace{3mm} b=\begin{pmatrix} 0 \\ 0 \end{pmatrix} ,\hspace{3mm} c_1 = \frac{1}{\sqrt{6}} \begin{pmatrix} -1 \\ 3 \end{pmatrix}, \hspace{3mm}  \text{and} \hspace{3mm} c_2 = \frac{1}{\sqrt{6}} \begin{pmatrix} 1 \\ 3 \end{pmatrix},
  \end{equation}
  we have
  \begin{equation}\label{eq:f6info2}
  q^{-\frac{1}{4}} f_6 (q) =  -\frac{1}{2} \sum_{\ell =1}^2 \sum_{{r} \in a_\ell +\mathbb{Z}^2} \rho_B({r})q^{Q({r})} e(B({r} ,b)).
  \end{equation}
 \end{lemma}

 \begin{proof}
 Proposition $3.7$ of \cite{BK-Multiplicative} states
 \begin{equation}
 \notag
 f_6 (q) = - \sum_{\substack{n\geq 0 \\ -n \leq j \leq n}} q^{3n^2 +3n +1 - j^2} - \sum_{\substack{n\geq 0 \\ -n \leq j \leq n-1}} q^{3n^2 - j^2 -j}.
 \end{equation}
A similar proof to that of Lemma \ref{lem:f1} gives \eqref{eq:f6info2}.
\end{proof}


\begin{lemma}\label{lem:f7}
  Setting $A= \left(\begin{smallmatrix} 6&0 \\0&-2 \end{smallmatrix} \right)$,
  \begin{equation}
  \notag
  a = \begin{pmatrix} \frac{1}{3} \\ 0 \end{pmatrix}, \hspace{5mm} b=\begin{pmatrix} \frac{1}{12} \\ \frac{1}{4} \end{pmatrix} , \hspace{5mm} c_1 = \frac{1}{\sqrt{6}} \begin{pmatrix} -1 \\ 3 \end{pmatrix},  \hspace{3mm}  \text{and} \hspace{3mm}c_2 = \frac{1}{\sqrt{6}} \begin{pmatrix} 1 \\ 3 \end{pmatrix},
  \end{equation}
  we have
  \begin{equation}\label{eq:f7info2}
  q^{\frac{1}{3}} f_7 (q) =  e^{-\pi i/3} \sum_{{r} \in a +\mathbb{Z}^2} \rho_B({r})q^{Q({r})} e(B({r} ,b)),
  \end{equation}
  where $\zeta_n := e(1/n)$.
 \end{lemma}

 \begin{proof}
 Proposition $3.8$ of \cite{BK-Multiplicative} states
 \begin{equation}
 \notag
 f_7 (q) = \sum_{\substack{n\geq 0 \\ -n \leq j \leq n}} (-1)^{n+j} q^{3n^2 + 2n - j^2} \left(1-q^{2n+1}\right) .
 \end{equation}
 In this case, 
 \begin{equation}\label{eq:f7eq1}
 3n^2 + 2n - j^2 = 3\left(n+\frac{1}{3}\right)^2 -j^2 -\frac{1}{3},
 \end{equation}
  while applying the change of variables $n \mapsto -n-1$ to \eqref{eq:f7eq1} gives $3n^2 + 4n +1 - j^2$. Therefore, taking ${r} = \left(\begin{smallmatrix} n \\ j \end{smallmatrix} \right)+\colvector{1/3}{0}$ for $n,j \in \mathbb{Z}$ leads to finding $e(B({r} ,b)) = (-1)^{n+j}e^{-\pi i /3}$, and a similar analysis to the other cases gives \eqref{eq:f7info2}.
 \end{proof}


\begin{lemma}\label{lem:f8}
  Setting $A= \left(\begin{smallmatrix} 12&0 \\0&-4 \end{smallmatrix} \right)$,
  \begin{equation}
  \notag
  a_1 =  \begin{pmatrix} \frac{1}{3} \\ 0 \end{pmatrix}, \hspace{3mm} a_2 =  \begin{pmatrix} \frac{1}{6} \\ \frac{1}{2} \end{pmatrix}, \hspace{3mm} b=\begin{pmatrix} 0 \\ 0 \end{pmatrix} ,\hspace{3mm} c_1 = \frac{1}{\sqrt{12}} \begin{pmatrix} -1 \\ 3 \end{pmatrix}, \hspace{3mm}  \text{and} \hspace{3mm} c_2 = \frac{1}{\sqrt{12}} \begin{pmatrix} 1 \\ 3 \end{pmatrix},
  \end{equation}
  we have
  \begin{equation}\label{eq:f8info2}
  q^{-\frac{1}{3}} f_8 (q) =  \sum_{\ell =1}^2 (-1)^{\ell+1} \sum_{{r} \in a_\ell +\mathbb{Z}^2} \rho_B({r})q^{Q({r})} e(B({r} ,b)).
  \end{equation}
  \end{lemma}
  
\begin{proof}
 Proposition $3.9$ of \cite{BK-Multiplicative} states
 \begin{equation}\label{eq:BKf8}
 f_8 (q) = -\sum_{\substack{n\geq 1 \\ -n \leq j \leq n-1}} q^{6n^2 -2n -2j^2 -2j}\left( 1+q^{4n}\right) 
   + \sum_{\substack{n\geq 0 \\ -n \leq j \leq n}} q^{6n^2 +4n +1 -2j^2}\left( 1+q^{4n+2}\right).
 \end{equation}
 Focusing on the left sum we note 
 \[
 6n^2 -2n -2j^2 -2j = 6\left(n-\frac{1}{6}\right)^2 -2\left(j+\frac{1}{2}\right)^2 +\frac{1}{3}.
 \]
  The change of variables $n \mapsto -n$ makes this expression $6n^2 +2n -2j^2 -2j$. (See the discussion found in the proof of Lemma \ref{lem:f8} for the rationale of our choice of $a_2$ here.) Meanwhile, in the right sum of \eqref{eq:BKf8} we find 
  \begin{equation}\label{eq:f8eq1}
  6n^2 +4n +1 -2j^2  = 6\left(n+\frac{1}{3}\right)^2 -2j^2 +\frac{1}{3}.
  \end{equation}
   Here, applying $n \mapsto -n-1$ to \eqref{eq:f8eq1} gives $6n^2 +8n +3 -2j^2$. A similar analysis to Lemma \ref{lem:f1} gives \eqref{eq:f8info2}. 
\end{proof}


\begin{lemma}\label{lem:LL}
  Setting $A= \left(\begin{smallmatrix} 4&0 \\0&-2 \end{smallmatrix} \right)$,
  \begin{equation}
  \notag
  a =  \begin{pmatrix} 0 \\ 0 \end{pmatrix}, \hspace{3mm} b=\begin{pmatrix} \frac{1}{8} \\ -\frac{1}{4} \end{pmatrix} ,\hspace{3mm} c_1 = \frac{1}{\sqrt{2}} \begin{pmatrix} -1 \\ 2 \end{pmatrix}, \hspace{3mm}  \text{and} \hspace{3mm} c_2 = \frac{1}{\sqrt{2}} \begin{pmatrix} 1 \\ 2 \end{pmatrix},
  \end{equation}
  we have
  \begin{equation}\label{eq:LLinfo2}
 LL(q) -\frac{1}{4}= -\frac{1}{2} \sum_{{r} \in a +\mathbb{Z}^2} \rho_B({r})q^{Q({r})} e(B({r} ,b)).
  \end{equation}
  \end{lemma}
  
\begin{proof}
 The series
 \[
 LL(q) := \sum_{n=1}^\infty \frac{(q)_{n-1} (-1)^n q^{\frac{1}{2}n(n+1)}}{(-q)_n}
 \]
 is originally found in \cite{Lovejoy}, and for $\lvert q \rvert <1$  we have
 \begin{equation}
 \notag
 LL(q) =\sum_{n=1}^\infty \sum_{-n <j\leq n} (-1)^{n+j+1} q^{2n^2 -j^2}
 \end{equation}
 (see (5.2) in \cite{LNR-Renorm2} and (2.11) in \cite{Lovejoy}). Note $n\mapsto -n$ preserves $2n^2 -j^2$. Additionally, for ${r} = \left(\begin{smallmatrix} n \\ j \end{smallmatrix} \right)$, $n,j \in \mathbb{Z}$, we have $e(B({r} ,b)) =(-1)^{n+j}$, and also
 \begin{equation}
 \notag
 \sum_{\substack{n-j =0 \\ 0<n+j}} (-1)^{n+j} q^{2n^2 -j^2} = \sum_{n\geq 1} q^{n^2} =\sum_{\substack{n+j =0 \\ n-j <0}} (-1)^{n+j} q^{2n^2 -j^2}.
 \end{equation}
 A similar analysis as in the proof of Lemma \ref{lem:f4} gives \eqref{eq:LLinfo2}. 
 \end{proof} 


 \begin{lemma}\label{lem:L}
  Setting $A= \left(\begin{smallmatrix} 4&0 \\0&-2 \end{smallmatrix} \right)$,
  \begin{equation}
  \notag
  a =  \begin{pmatrix} 0 \\ 0 \end{pmatrix}, \hspace{3mm} b= \begin{pmatrix} \frac{1}{8} \\ 0 \end{pmatrix} ,\hspace{3mm} c_1 = \frac{1}{\sqrt{2}} \begin{pmatrix} -1 \\ 2 \end{pmatrix}, \hspace{3mm}  \text{and} \hspace{3mm} c_2 = \frac{1}{\sqrt{2}} \begin{pmatrix} 1 \\ 2 \end{pmatrix},
  \end{equation}
  we have
  \begin{equation}\label{eq:Linfo2}
 L(q) -\frac{1}{4}  = -\frac{1}{2} \sum_{{r} \in a +\mathbb{Z}^2} \rho_B({r})q^{Q({r})} e(B({r} ,b)).
  \end{equation}
  \end{lemma}
  
\begin{proof}
 The series
 \[
 L(q) := \sum_{n=1}^\infty \frac{q^n  \left(q^2 ;q^2\right)_{n-1}}{\left(-q^2 ;q^2 \right)_n}
 \]
 is considered in \cite{LNR-Renorm2} (see display (5.3)). By Proposition 5.5 of \cite{LNR-Renorm2} we have $L(-q) =LL(q)$, and the remainder of the proof of \eqref{eq:Linfo2} follows from the previous lemma and the equality
 \begin{equation}
 \notag
 \sum_{\substack{n-j =0 \\ 0<n+j}} (-1)^{r} q^{2n^2 -j^2} = \sum_{n\geq 1} (-1)^n q^{n^2} =\sum_{\substack{n+j =0 \\ n-j <0}} (-1)^{r} q^{2n^2 -j^2}.
 \end{equation}
\end{proof}

We conclude this section by summarizing the above choices of parameters for each function in a table.
\begin{table}[h!]
\begin{center}
\addtolength{\tabcolsep}{1mm}
\renewcommand{\arraystretch}{1.2}
\begin{tabular}{|c|c|c|c|c|c|c|}
\hline
\hline
\text{function} & \text{matrix $A$}& \text{parameters $a,a_1,a_2$} & \text{parameter $b$} & \text{parameter $c_1$} & \text{parameter $c_2$}     \\
\hline
\hline
$f_1$ &  $\left(\begin{smallmatrix} 8&0 \\0&-4 \end{smallmatrix} \right)$ &  $ a_1 =  \left(\begin{smallmatrix} \frac{5}{8} \\ \frac{1}{2} \end{smallmatrix}\right), \hspace{3mm} a_2 =  \left(\begin{smallmatrix} \frac{1}{8} \\ 0 \end{smallmatrix}\right)$ & $\left(\begin{smallmatrix} 0 \\ 0 \end{smallmatrix}\right)$ & $\left(\begin{smallmatrix} -\frac{1}{2} \\ 1 \end{smallmatrix}\right)$ &  $\left(\begin{smallmatrix} \frac{1}{2} \\ 1 \end{smallmatrix}\right)$  \\
\hline
$f_2$ &  $\left( \begin{smallmatrix}8&0 \\0&-4 \end{smallmatrix} \right)$ &  $ a_1 =  \left(\begin{smallmatrix}  \frac{3}{8} \\ 0 \end{smallmatrix}\right), \hspace{3mm} a_2 =  \left(\begin{smallmatrix} \frac{1}{8} \\ \frac{1}{2} \end{smallmatrix}\right)$ & $\left(\begin{smallmatrix} 0 \\ 0 \end{smallmatrix}\right)$ & $\left(\begin{smallmatrix} -\frac{1}{2} \\ 1 \end{smallmatrix}\right)$ &  $\left(\begin{smallmatrix} \frac{1}{2} \\ 1 \end{smallmatrix}\right)$  \\
\hline
$f_3$ &  $\left(\begin{smallmatrix} 4&0 \\0&-2 \end{smallmatrix} \right)$ &  $ a =  \left(\begin{smallmatrix} \frac{1}{2} \\ 0 \end{smallmatrix}\right)$ & $\left(\begin{smallmatrix} 0 \\ \frac14 \end{smallmatrix}\right)$ & $\frac{1}{\sqrt{2}} \left(\begin{smallmatrix} -1 \\ 2 \end{smallmatrix}\right)$ &  $\frac{1}{\sqrt{2}} \left(\begin{smallmatrix} 1 \\ 2 \end{smallmatrix}\right)$  \\
\hline
$f_4$ &  $\left(\begin{smallmatrix} 4&0 \\0&-2 \end{smallmatrix} \right)$ &  $ a =  \left(\begin{smallmatrix} 0 \\ 0 \end{smallmatrix}\right)$ & $\left(\begin{smallmatrix} 0 \\ \frac14 \end{smallmatrix}\right)$ & $ \frac{1}{\sqrt{2}}\left(\begin{smallmatrix} -1\\ 2 \end{smallmatrix}\right)$ &  $ \frac{1}{\sqrt{2}}\left(\begin{smallmatrix} 1 \\ 2 \end{smallmatrix}\right)$  \\
\hline
$f_5$ &  $\left(\begin{smallmatrix} 3&0 \\0&-1 \end{smallmatrix} \right)$ &  $ a=  \left(\begin{smallmatrix} \frac{1}{2} \\ \frac{1}{2}\end{smallmatrix}\right)$ & $\left(\begin{smallmatrix} 0 \\ 0 \end{smallmatrix}\right)$ & $\frac{1}{\sqrt{3}}\left(\begin{smallmatrix} -1\\ 3 \end{smallmatrix}\right)$ &  $\frac{1}{\sqrt{3}}\left(\begin{smallmatrix} 1 \\ 3 \end{smallmatrix}\right)$  \\
\hline
$f_6$ &  $\left(\begin{smallmatrix} 6&0 \\0&-2 \end{smallmatrix} \right)$ &  $ a_1 =  \left(\begin{smallmatrix} \frac{1}{2}  \\ 0 \end{smallmatrix}\right), \hspace{3mm} a_2 =  \left(\begin{smallmatrix} 0  \\ \frac{1}{2} \end{smallmatrix}\right)$ & $\left(\begin{smallmatrix} 0 \\ 0 \end{smallmatrix}\right)$ & $\frac{1}{\sqrt{6}}\left(\begin{smallmatrix} -1 \\ 3 \end{smallmatrix}\right)$ &  $\frac{1}{\sqrt{6}}\left(\begin{smallmatrix} 1 \\ 3 \end{smallmatrix}\right)$  \\
\hline
$f_7$ &  $\left(\begin{smallmatrix} 6&0 \\0&-2 \end{smallmatrix} \right)$ &  $ a =  \left(\begin{smallmatrix}\frac{1}{3} \\ 0 \end{smallmatrix}\right)$ & $\left(\begin{smallmatrix} \frac{1}{12} \\ \frac{1}{4}  \end{smallmatrix}\right)$ & $\frac{1}{\sqrt{6}}\left(\begin{smallmatrix} -1 \\ 3 \end{smallmatrix}\right)$ &  $\frac{1}{\sqrt{6}}\left(\begin{smallmatrix} 1 \\ 3 \end{smallmatrix}\right)$  \\
\hline
$f_8$ &  $\left(\begin{smallmatrix} 12&0 \\0&-4 \end{smallmatrix} \right)$ &  $ a_1 =  \left(\begin{smallmatrix} \frac{1}{3} \\ 0 \end{smallmatrix}\right), \hspace{3mm} a_2 =  \left(\begin{smallmatrix} -\frac{1}{6} \\ \frac{1}{2} \end{smallmatrix}\right)$ & $\left(\begin{smallmatrix} 0 \\ 0 \end{smallmatrix}\right)$ & $\frac{1}{\sqrt{12}}\left(\begin{smallmatrix} -1 \\ 3 \end{smallmatrix}\right)$ &  $\frac{1}{\sqrt{12}}\left(\begin{smallmatrix}1\\ 3 \end{smallmatrix}\right)$  \\
\hline
$LL$ &  $\left(\begin{smallmatrix} 4&0 \\0&-2 \end{smallmatrix} \right)$ &  $ a=  \left(\begin{smallmatrix} 0\\ 0 \end{smallmatrix}\right))$ & $\left(\begin{smallmatrix} \frac{1}{8} \\ -\frac{1}{4} \end{smallmatrix}\right)$ & $\frac{1}{\sqrt{2}} \left(\begin{smallmatrix} -1 \\ 2 \end{smallmatrix}\right)$ &  $\frac{1}{\sqrt{2}} \left(\begin{smallmatrix} 1 \\ 2 \end{smallmatrix}\right)$  \\
\hline
$L$ &  $\left(\begin{smallmatrix} 4&0 \\0&-2 \end{smallmatrix} \right)$ &  $ a =  \left(\begin{smallmatrix} 0 \\ 0 \end{smallmatrix}\right)$ & $\left(\begin{smallmatrix} \frac{1}{8} \\ 0 \end{smallmatrix}\right)$ & $\frac{1}{\sqrt{2}}\left(\begin{smallmatrix} -1\\ 2 \end{smallmatrix}\right)$ &  $\frac{1}{\sqrt{2}}\left(\begin{smallmatrix} 1\\ 2 \end{smallmatrix}\right)$  \\
\hline
\end{tabular}\\[1mm]
\caption{Parameters of indefinite theta series corresponding to the $q$-series $f_1 ,\dots ,f_8, LL, L$.}
\label{table}
\end{center}
\end{table}

\subsection{Proof of Theorem \ref{mainthm}}\label{sec:proofb} \ 

We are now in position to apply Lemma \ref{Zlem} and Theorem \ref{Zthm} to establish Theorem \ref{mainthm}.
\begin{proof}
 We first consider $f_1$. Recall $A,B,a_1,a_2,b$, $c_1$, and $c_2$ associated to $f_1$ from Lemma \ref{lem:f1}. Taking $\gamma = \left(\begin{smallmatrix} 3&2 \\4&3 \end{smallmatrix} \right)$, $\lambda_1 = \left(\begin{smallmatrix} 3 \\ 4 \end{smallmatrix} \right)$, $\lambda_2 = \left(\begin{smallmatrix} 1 \\ 1 \end{smallmatrix} \right)$, $\mu = \left(\begin{smallmatrix} 0 \\ 0 \end{smallmatrix} \right)$, and using Lemma \ref{Zlem}, we have
 \begin{equation}\notag
  \varphi^{c_1}_{a_1 ,b} = \varphi^{\gamma c_1}_{\gamma a_1 , \gamma b} = \varphi^{c_2}_{-a_2 +\lambda_1 ,-b +\mu} = e(B(-a_2 ,\mu)) \varphi^{c_2}_{-a_2 ,-b} = \varphi^{c_2}_{a_2 ,b}  
  \end{equation}
  and
  \begin{equation}\notag
  \varphi^{c_1}_{a_2 ,b} = \varphi^{\gamma c_1}_{\gamma a_2 , \gamma b} = \varphi^{c_2}_{-a_1 +\lambda_2 ,-b +\mu} = e (B(-a_1 ,\mu)) \varphi^{c_2}_{-a_1 ,-b} = \varphi^{c_2}_{a_1 ,b}.
  \end{equation}
  Therefore $(\varphi^{c_1}_{a_1 ,b} -\varphi^{c_2}_{a_1 ,b})+(\varphi^{c_1}_{a_2 ,b} -\varphi^{c_2}_{a_2 ,b})=0$. 

Next, it is easy to show that $a_j +\frac{1}{2}A^{-1}A^\star \in A^{-1} \mathbb{Z}^2$. Finally, 
one may check from Theorem~\ref{Zthm} that that the space
 \[ \text{span}_\mathbb{C} \{ \widehat{\Phi}_{a+p_1,b+b_2}, \widehat{\Phi}_{b +p_1, a+p_2} \mid p_1,p_2 \in A^{-1} \mathbb{Z}^2 \} \]
is closed under the transformations $\tau \mapsto -1/\tau$, and since $a_j +\frac{1}{2}A^{-1}A^\star \in A^{-1} \mathbb{Z}^2$ it is closed under $\tau \mapsto 1+\tau$ as well. Hence, the result of Theorem \ref{mainthm} for $f_1$ now follows from Theorem \ref{Zthm}.
  
  The proofs of the remaining functions follow similarly once the appropriate parameters are found. We note that for some of them, we get other relations between $c_1,c_2$, such as
  $\gamma a_1=a_2+\lambda_1$, or even $\gamma a_1=-a_1+\lambda_1$. The completion terms all can be shown by simple arguments to still vanish, and in particular we point the reader to the discussion of the proof of (14) in \cite{Zwegers-Maass} for a related argument.  
   We provide these inputs for each function in Table \ref{table} below, while the necessary equalities involving the $\varphi_{a,b}^{c_j}$ are left to the reader.

\begin{table}[h!]
\begin{center}
\addtolength{\tabcolsep}{1mm}
\renewcommand{\arraystretch}{1.2}
\begin{tabular}{|c|c|c|c|c|c|c|}
\hline
\hline
\text{function} & \text{matrix $\gamma$}& \text{parameter $\lambda$} & \text{parameter $\mu$}    \\
\hline
\hline
$f_1$ & $\left(\begin{smallmatrix} 3&2 \\4&3 \end{smallmatrix} \right)$ & $\lambda_1 =  \left(\begin{smallmatrix} 3 \\ 4 \end{smallmatrix} \right)$, $\lambda_2 = \left(\begin{smallmatrix} 1 \\ 1 \end{smallmatrix} \right) $ & $\left(\begin{smallmatrix} 0 \\ 0 \end{smallmatrix} \right)$   \\
\hline
$f_2$ & $\left(\begin{smallmatrix} 3&2 \\4&3 \end{smallmatrix} \right)$ & $\lambda_1 =  \left(\begin{smallmatrix} 1 \\ 1 \end{smallmatrix} \right)$, $\lambda_2 =  \left(\begin{smallmatrix} 1 \\ 2 \end{smallmatrix} \right)$ & $\left(\begin{smallmatrix} 0 \\ 0 \end{smallmatrix} \right)$   \\
\hline
$f_3$ & $\left(\begin{smallmatrix} 3&2 \\4&3 \end{smallmatrix} \right)$ & $\left(\begin{smallmatrix} 1 \\ 2 \end{smallmatrix} \right)$ & $\frac{1}{2}\left(\begin{smallmatrix} 1 \\ 1 \end{smallmatrix} \right)$   \\
\hline
$f_4$ & $\left(\begin{smallmatrix} 3&2 \\4&3 \end{smallmatrix} \right)$ & $\left(\begin{smallmatrix} 0 \\ 0 \end{smallmatrix} \right)$ & $\frac{1}{2}\left(\begin{smallmatrix} 1 \\ 1 \end{smallmatrix} \right)$   \\
\hline
$f_5$ & $\left(\begin{smallmatrix} 2&1 \\3&2 \end{smallmatrix} \right)$ & $\left(\begin{smallmatrix} 1 \\ 2 \end{smallmatrix} \right)$ & $\left(\begin{smallmatrix} 0 \\ 0 \end{smallmatrix} \right)$   \\
\hline
$f_6$ & $\left(\begin{smallmatrix} 2&1 \\3&2 \end{smallmatrix} \right)$ & $\lambda_1 =  \left(\begin{smallmatrix} 1 \\ 1 \end{smallmatrix} \right)$, $\lambda_2 =  \left(\begin{smallmatrix} 0 \\ 1 \end{smallmatrix} \right)$ & $\left(\begin{smallmatrix} 0 \\ 0 \end{smallmatrix} \right)$   \\
\hline
$f_7$ & $\left(\begin{smallmatrix} 2&1 \\3&2 \end{smallmatrix} \right)$ & $\left(\begin{smallmatrix} 1 \\ 1 \end{smallmatrix} \right)$ & $\frac{1}{2}\left(\begin{smallmatrix} 1 \\ 2 \end{smallmatrix} \right)$   \\
\hline
$f_8$ & $\left(\begin{smallmatrix} 2&1 \\3&2 \end{smallmatrix} \right)$ & $\lambda_1 =  \left(\begin{smallmatrix} 1 \\ 1 \end{smallmatrix} \right)$, $\lambda_2 =  \left(\begin{smallmatrix} 1 \\ 2 \end{smallmatrix} \right)$ & $\left(\begin{smallmatrix} 0 \\ 0 \end{smallmatrix} \right)$   \\
\hline
$LL$ & $\left(\begin{smallmatrix} 3&2 \\4&3 \end{smallmatrix} \right)$ & $\left(\begin{smallmatrix} 0 \\ 0 \end{smallmatrix} \right)$ & $-\frac{1}{2}\left(\begin{smallmatrix} 0 \\ 1 \end{smallmatrix} \right)$   \\
\hline
$L$ & $\left(\begin{smallmatrix} 3&2 \\4&3 \end{smallmatrix} \right)$ & $\left(\begin{smallmatrix} 0 \\ 0 \end{smallmatrix} \right)$ & $\frac{1}{2}\left(\begin{smallmatrix} 1 \\ 1 \end{smallmatrix} \right)$   \\
\hline
\end{tabular}\\[1mm]
\caption{Parameters for the functions $f_1 ,\dots ,f_8, LL, L$.}
\label{table}
\end{center}
\end{table}
The final result for all functions, and thus Theorem \ref{mainthm}, now follows from Theorem \ref{Zthm} and the discussion above.
\end{proof}
 
\begin{rmk*}
 A similar analysis can be performed for $W_1$ using the parameters $A= \left(\begin{smallmatrix} 4&0 \\0&-2 \end{smallmatrix} \right)$,
  \begin{equation}
  \notag
  a = \begin{pmatrix} \frac{1}{4} \\ 0 \end{pmatrix}, \hspace{3mm} b=\begin{pmatrix} \frac{1}{8} \\ -\frac{1}{4} \end{pmatrix} ,\hspace{3mm} c_1 = \frac{1}{\sqrt{2}} \begin{pmatrix} -1 \\ 2 \end{pmatrix}, \hspace{3mm}  \text{and} \hspace{3mm} c_2 = \frac{1}{\sqrt{2}} \begin{pmatrix} 1 \\ 2 \end{pmatrix},
  \end{equation}
  and taking $\gamma = \left(\begin{smallmatrix} 3&2\\4&3 \end{smallmatrix} \right)$, $\lambda =  \left(\begin{smallmatrix} 1 \\ 1 \end{smallmatrix} \right)$, and $\mu = -\frac{1}{2}\left(\begin{smallmatrix} 0 \\ 1 \end{smallmatrix} \right)$.
  
 The parameters for $W_2$ are $A= \left(\begin{smallmatrix} 4&0 \\0&-2 \end{smallmatrix} \right)$,
  \begin{equation}
  \notag
  a = -\begin{pmatrix} \frac{1}{4} \\ \frac{1}{2} \end{pmatrix}, \hspace{3mm} b=\begin{pmatrix} \frac{1}{8} \\ 0 \end{pmatrix} ,\hspace{3mm} c_1 = \frac{1}{\sqrt{2}} \begin{pmatrix} -1 \\ 2 \end{pmatrix}, \hspace{3mm}  \text{and} \hspace{3mm} c_2 = \frac{1}{\sqrt{2}} \begin{pmatrix} 1 \\ 2 \end{pmatrix},
  \end{equation}
  along with $\gamma = \left(\begin{smallmatrix} 3&2\\4&3 \end{smallmatrix} \right)$, $\lambda =  -\left(\begin{smallmatrix} 2 \\ 3 \end{smallmatrix} \right)$, and $\mu = \frac{1}{2}\left(\begin{smallmatrix} 1 \\ 1 \end{smallmatrix} \right)$.
\end{rmk*}


\end{document}